\documentclass[12pt]{article}

\usepackage[T1]{fontenc}
\usepackage{textcomp}
\usepackage{mathtools}
\usepackage{stmaryrd}

\usepackage{amssymb}
\usepackage{amsmath}
\usepackage{latexsym}
\usepackage{amscd}
\usepackage{amsthm}

\usepackage[dvipdfm]{graphicx}

\newcommand{\Q}{\Bbb Q}
\newcommand{\Z}{\Bbb Z}

\newcommand{\ga}{{\rm Gal}}

\newcommand{\ca}[1]{{\mathcal #1}}
\newcommand{\lr}[1]{\langle #1 \rangle}

\newcommand{\fulltoday}{\number\day\space \ifcase\month\or
    January\or February\or March\or April\or May\or June\or
    July\or August\or September\or October\or November\or December\fi
    \space\number\year}

\theoremstyle{plain}

\newtheorem{thm}{\indent\bf Theorem}%[section]
\newtheorem{lem}{\indent\bf Lemma}%[section]
\newtheorem*{thmA}{\indent\bf Theorem A}%[section]

\theoremstyle{definition}

\newtheorem*{ack}{\indent\bf Acknowledgments}
\newtheorem*{defi}{\indent\bf Definition}

\begin{document}
\title{
On automorphisms of some semidirect product groups and ranks of Iwasawa modules
}

\author{
Satoshi Fujii\thanks{
Faculty of Education, Shimane University, 
1060 Nishikawatsucho, Matsue, Shimane, 690--8504, Japan. 
e-mail : {\tt fujiisatoshi@edu.shimane-u.ac.jp}, 
{\tt mophism@gmail.com}
}
}
\date{
%\today
}
\maketitle
%\begin{center}
%{\Large\bf On minus quotients of ideal class groups of cyclotomic fields}
%\footnote{2000 \textit{Mathematics Subject Classification}. 11R18, 11R23, 11R29.}
%\end{center}
%\vspace{10pt}
%\begin{center}
%{\bf By Satoshi FUJII
%\footnote{
%Faculty of Education, Shimane University, 
%1060 Nishikawatsucho, Matsue, Shimane, 690--8504, Japan. 
%e-mail : {\tt fujiisatoshi@edu.shimane-u.ac.jp}
%}}
%\end{center}
%\vspace{20pt}

\begin{abstract}
Let $p$ be an odd prime number and $k$ an imaginary quadratic field in which $p$ does not split. 
Based on their heuristic, 
Kundu and Washington posed a question which asks whether $\lambda$- and $\mu$-invariant of 
the anti-cyclotomic $\Z_p$-extension $k_{\infty}^a$ of $k$ are always trivial. 
Also, 
if $k_{\infty}^a/k$ is totally ramified, 
for $n\geq 1$, 
they showed that the $p$-part of the ideal class group of the 
$n$th layer of the anti-cyclotomic $\Z_p$-extension of $k$ is not cyclic. 
In this article, 
inspired by their paper, 
we study anti-cyclotomic like $\Z_p$-extensions, 
extending both the above question and 
Kundu-Washington's result. 
We show that the values of $\lambda$ of certain anti-cyclotomic like 
$\Z_p$-extensions are always even. 
We also show the $p$-part of the ideal class groups of certain 
anti-cyclotomic like $\Z_p$-extensions of CM-fields are always not cyclic. 
\end{abstract}

%\vspace{-10pt}
\section{Introduction}

Let $p$ be a fixed prime number. 
A Galois extension is called a $\Z_p$-extension 
if its Galois group is topologically isomorphic to the additive group of the ring of 
$p$-adic integers $\Z_p$. 
Let $k/\Q$ be a finite extension and $K/k$ a $\Z_p$-extension. 
For each non-negative integer $n$, 
there is the unique intermediate field $k_n$, 
called the $n$th layer, 
of $K/k$ such that $[k_n:k]=p^n$. 
For a finite extension $F/\Q$, 
let $A_F$ denote the $p$-part of the ideal class group of $F$. 
For a $\Z_p$-extension $K/k$, 
let $X_K$ denote the Galois group of the maximal unramified abelian pro-$p$ extension 
$L_K/K$. 
The module $X_K$ is also defined to be the projective limit $\varprojlim_n A_{k_n}$ with respect to norm maps. 
By Iwasawa's class number formula, 
there are non-negative integers $\lambda$, $\mu$ and an integer $\nu$ depending only on $K/k$ 
such that $\# A_{k_n}=p^{\lambda n+\mu p^n+\nu}$ for all sufficiently large $n$. 
The integers $\lambda$ and $\mu$ are structure invariants of $X_K$ as a Galois module. 
In particular, 
it is known that 
$\lambda=\dim_{\Q_p}\Q_p\otimes_{\Z_p}X_K$, 
and that $\lambda=\mu=0$ if and only if $X_K$ is finite. 
Here we denote by $\Q_p$ the $p$-adic number field.

Suppose that $p$ is odd. 
For each imaginary quadratic field $k$, 
there is the unique $\Z_p$-extension $k_{\infty}^a/k$ such that 
$k_{\infty}^a/\Q$ is a non-abelian Galois extension. 
The extension $k_{\infty}^a/k$ is called the anti-cyclotomic $\Z_p$-extension of $k$. 
In \cite{Kundu-Washington}, 
based on their heuristic, 
Kundu and Washington posed a question: 
Is $\lambda= \mu = 0$ always true for the anti-cyclotomic 
$\Z_p$-extension of an imaginary quadratic field $k$ in which $p$ does not split?
For each non-negative integer $n$ let $k_n^a$ be the $n$th layer of $k_{\infty}^a/k$. 
Kundu and Washington also showed the following result. 

\begin{thmA}\label{A}
Let $p$ be an odd prime number and $k$ an imaginary quadratic field in which $p$ does not split. 
Suppose that $k_{\infty}^a/k$ is totally ramified at the unique prime above $p$. 
If $A_k\neq 0$ then $A_{k_n^a}$ is not cyclic for $n\geq 1$. 
\end{thmA}

In the present article, 
by giving tentatively the following definition, 
we study for more general settings than anti-cyclotomic $\Z_p$-extensions of  imaginary quadratic fields.

\begin{defi}
Let $p$ be an odd prime number. 
Let $F/\Q$ be a finite extension and $k/F$ a quadratic extension. 
A $\Z_p$-extension $k_{\infty}^a/k$ is anti-cyclotomic like with respect to $k/F$ if 
$k_{\infty}^a/F$ is a non-abelian Galois extension. 
\end{defi}

Let $k$ and $F$ be totally real fields such that $k/F$ is a quadratic extension. 
If Leopoldt's conjecture holds true for $p$ and $k$, 
then $k$ has only the cyclotomic $\Z_p$-extension, 
and hence there are no anti-cyclotomic like $\Z_p$-extensions with respect to $k/F$. 
Also, 
let $k$ be a CM-field and $k^+$ the maximal totally real subfield of $k$. 
If $[k:\Q]>2$ then there are infinitely many anti-cyclotomic like $\Z_p$-extensions 
with respect to $k/k^+$. 

Let $k_{\infty}^a/k$ be an anti-cyclotomic like $\Z_p$-extension with respect to a quadratic extension $k/F$. 
Let $J$ be the generator of $\ga(k/F)$. 
Since $k_{\infty}^a/F$ is a non-abelian Galois extension, 
$J$ acts on $\ga(k_{\infty}^a/k)$ as the inverse. 
For each non-negative integer $n$, 
denote by $k_n^a$ the $n$th layer of $k_{\infty}^a/k$. 
The first result of this article 
considering $\lambda$-invariants. 

\begin{thm}\label{1}
Let $p$ be an odd prime number, 
$k/F$ a quadratic extension over a finite extension $F/\Q$ and 
$k_{\infty}^a/k$ an anti-cyclotomic like $\Z_p$-extension with respect to $k/F$. 
Suppose that the prime $p$ does not split in $k_{\infty}^a/\Q$. 
Then we have $\lambda\equiv 0\bmod{2}$.
\end{thm}

Theorem \ref{1} asserts that odd positive integers, 
half of all positive integers, 
actually do not appear as $\lambda$ if $p$ does not split in $k_{\infty}^a/\Q$. 
We remark that the values of $\lambda$ of the cyclotomic $\Z_p$-extensions are often odd. 
We also remark that, 
when $p=3$ or $5$, 
there is an imaginary quadratic field $k$ in which $p$ splits 
such that $\lambda$ of $k_{\infty}^a/k$ is $p$, 
see examples below in Theorem $2$ and Theorem $4.2$ of \cite{Fujii2013}. 
To prove Theorem \ref{1}, 
the argument of the proof of Theorem $3$ of \cite{Carroll-Kisilevsky}, 
based on the structure theorem of $\Z_p\llbracket T \rrbracket$-modules, 
can be applied.  
In this article, 
we will prove Theorem \ref{1} using a group-theoretical method which can be seen as a 
variant of Lemma $6.6$ of \cite{Kundu-Washington}, 
see Lemma \ref{lem2} in section $2$. 

The second result is a generalization of Theorem A for CM-fields.  

\begin{thm}\label{2}
Let $p$ be an odd prime number, 
$k$ a CM-field and $k^+$ the maximal totally real subfield of $k$. 
Let $k_{\infty}^a/k$ be an anti-cyclotomic like $\Z_p$-extension with respect to $k/k^+$ 
such that $p$ does not split in $k_{\infty}^a/\Q$. 
Suppose that $k$ contains no primitive $p$th roots of unity, 
$A_k\neq 0$ and that $A_{k^+}=0$. 
Then $A_{k_n^a}$ is not cyclic for $n\geq 1$. 
\end{thm}

If $p=3$ and $k=\Q(\sqrt{-3})$ then $A_k=0$. 
If $k$ is an imaginary quadratic field then $k^+=\Q$, 
and hence $A_{k^+}=A_{\Q}=0$. 
Thus we can say that Theorem \ref{2} is a generalization of Theorem A 
for CM-fields. 
We will see that the proof of Theorem A can be applied almost directly to the proof of Theorem \ref{2}, 
we further consider the structure of the ideal class groups and the 
unit groups as $\ga(k/k^+)$-modules.  

The rest of this section consists of affirmations of some notations and fundamental facts. 
For a finite cyclic group $G_1$ and a $G_1$-module $M$, 
let $\hat{H}^i(G_1,M)$ be the $i$th Tate cohomology group for $i=-1,0$.  
For a group $G$ and a $G$-module $N$, 
put $N^G=\{x\in N\mid gx=x\;\mbox{for all }g\in G\}$. 
Let $p$ be a prime number. 
For an algebraic extension $K/\Q$, 
let $L_K/K$ be the maximal unramified abelian pro-$p$ extension and put 
$X_K=\ga(L_K/K)$. 
If $K/F$ is a Galois extension, 
it follows that $L_K/F$ is also a Galois extension. 
We then define an action of $\ga(K/F)$ on $X_K$ as follows.  
Let $g\in \ga(K/F)$, 
$x\in X_K$ and denote by $\tilde{g}\in \ga(L_K/F)$ an extension of $g$. 
Then $\ga(K/F)$ acts on $X_K$ as the inner automorphism $g(x)=\tilde{g}x\tilde{g}^{-1}$. 
When $K/\Q$ is a finite extension, 
by class field theory, 
the Artin map induces an isomorphism $A_K\simeq X_K$ as $\ga(K/F)$-modules. 
Let $\langle J \rangle$ be a cyclic group of order $2$. 
For an odd prime number $p$ and a $\Z_p[\langle J \rangle]$-module $M$, 
put $M^-=\frac{1-J}{2}M$. 
Also, 
for a $\Z[\langle J \rangle]$-module $M$, put $M_p^-=(\Z_p\otimes_{\Z} M)^-$. 
It is known that the functor $M\mapsto M_p^-$ is exact. 
For a positive integer $m$, 
let $\mu_m$ be the group of all $m$th roots of unity.

\section{
Automorphisms of semidirect products of some pro-$p$ abelian groups
}

In this section, 
let $p$ be an odd prime number. 
Let $G_1=\lr{\tau}$ be a cyclic group of order $p$ generated by $\tau$, 
and let $A_1$ be a cyclic group of order $p^{u+1}$ for some positive integer  $u$. 
Suppose that $G_1$ acts on $A_1$ non-trivially. 
Let $G=A_1\rtimes G_1$. 
Kundu and Washington proved the following. 

\begin{lem}[Lemma 6.6 of \cite{Kundu-Washington}]\label{lem1}
There is no automorphism $\phi$ of $G$ such that 
$\phi(\tau)=y\tau^{-1}$ with $y\in A_1$. 
\end{lem}

Kundu and Washington have used a presentation of $G$ by $2\times 2$ matrices. 
We prefer to give the proof briefly here. 

\begin{proof}
Suppose that such an automorphism $\phi$ of $G$ exists. 
Let $x$ be a generator of $A_1$. 
We may assume that $\tau x\tau^{-1}=x^{1+p^u}$. 
We remark that 
$(1+p^u)^r\equiv 1+rp^u\bmod{p^{u+1}}$ for all $r\in \Z$. 
Put $y=x^a$ and $\phi(x)=x^b\tau^c$ for some $a,b,c\in \Z$. 
Since $\tau^p=1$ and
\begin{align*}
1
&
=
\phi(\tau)^p
\\
&
=
(x^a\tau^{-1})^p
\\
&
=
x^a\tau^{-1}x^a\tau \tau^{-2}x^a\tau^2 \cdots \tau^{-(p-1)}x^a\tau^{p-1}\tau^{-p}
\\
&
=
x^ax^{(1+p^u)^{-1}a}x^{(1+p^u)^{-2}a}\cdots x^{(1+p^u)^{-(p-1)}a}
\\
&
=
x^{a\sum_{t=0}^{p-1}(1-tp^u)}
\\
&
=
x^{ap},
\end{align*}
we have $a\equiv 0\bmod{p^u}$. 
Also, 
since $x^{p^u}\neq 1$ and
$$
1
\neq 
\phi(x)^{p^u}
=
(x^b\tau^c)^{p^u}
=
x^{b\sum_{t=0}^{p^u-1}(1+ctp^u)}
=
x^{bp^u}, 
$$
we have $b\not\equiv 0\bmod{p}$. 
We then have
\begin{align*}
\phi(\tau x\tau^{-1})
&
=
\phi(\tau)\phi(x)\phi(\tau)^{-1}
\\
&
=
x^a\tau^{-1}x^b\tau^c\tau x^{-a}
\\
&
=
x^a\tau^{-1}x^b\tau\tau^c x^{-a}\tau^{-c}\tau^c
\\
&
=
x^ax^{(1+p^u)^{-1}b}x^{-(1+p^u)^ca}\tau^c
\\
&
=
x^{(1-p^u)b+a-a(1+cp^u)}\tau^c
\\
&
=
x^{(1-p^u)b-acp^u}\tau^c
\\
&
=
x^{(1-p^u)b}\tau^c
.
\end{align*}
On the other hand, 
we have
$$
\phi(\tau x\tau^{-1})
=
\phi(x^{1+p^u})
=
(x^b\tau^c)^{1+p^u}
=
x^{b\sum_{t=0}^{p^u}(1+tcp^u)}\tau^{(1+p^u)c}
=
x^{(1+p^u)b}\tau^c.
$$
Thus it follows that $(1-p^u)b\equiv (1+p^u)b\bmod{p^{u+1}}$. 
Since $p$ is an odd prime number and $b\not\equiv 0\bmod{p}$, 
this congruence does not hold. 
\end{proof}

Here, 
we give a variant of the above lemma to semidirect products of certain pro-$p$ abelian groups. 
Let $\Gamma$ be a pro-$p$ groups which is isomorphic to $\Z_p$ with a topological generator 
$\gamma$. 
Let $X$ be a pro-$p$ abelian group isomorphic to $\Z_p^r$ for some $r>0$. 
Let $\{x_1,\cdots, x_r\}$ be a basis of $X$. 
Suppose that $\Gamma$ acts on $X$ and that $X^{\Gamma}=0$. 
Put $\ca{G}=X\rtimes \Gamma$. 

\begin{lem}\label{lem2}
Here we regard $\mu_{p-1}\subseteq \Z_p^{\times}$. 
Let $\zeta \in \mu_{p-1}$ and $d$ the order of $\zeta$. 
If there is a topological automorphism $\psi$ of $\ca{G}$ such that $\psi(X)= X$ and that 
$\psi(\gamma)=x\gamma^{\zeta}$ with $x\in X$, 
then $r\equiv 0\bmod{d}$. 
\end{lem}

\begin{proof}
If $\zeta=1$ then $d=1$, 
and hence there is no what to do. 
Let $\zeta\neq 1$. 
Suppose that there is such an automorphism $\psi$. 
Let $(\alpha_{ij}),(\delta_{ij})$ be invertible $r\times r$ matrices with entries in $\Z_p$ 
such that 
$\gamma x_j\gamma^{-1}=\prod_{i=1}^rx_i^{\alpha_{ij}}$ 
and $\psi(x_j)=\prod_{i=1}^rx_i^{\delta_{ij}}$. 
Since $\Gamma$ is a pro-$p$ group, 
we can define $(\alpha_{ij})^{\zeta}$. 
Let $\eta_1,\cdots,\eta_t$ be all of distinct eigen values of $(\alpha_{ij})$ with multiplicities $m_1,\cdots ,m_t$. 
Thus $m_1+\cdots +m_t=r$. 
By our assumption that $\Gamma$ is a pro-$p$ group and $X^{\Gamma}=0$, 
we have $\eta_i \equiv 1 \bmod{\pi}$ and $\eta_i\neq 1$ for all $1\leq i\leq t$, 
here, 
$\pi$ denotes a prime element of $\Q_p(\eta_i\mid 1\leq i\leq t)$. 
It also holds that 
all of distinct eigen values of $(\alpha_{ij})^{\zeta}$ are 
$\eta_1^{\zeta},\cdots, \eta_t^{\zeta}$ with multiplicities $m_1,\cdots,m_t$. 
Indeed, 
we can easily see that $\eta_1^{\zeta},\cdots, \eta_t^{\zeta}$ are all of eigen values of 
$(\alpha_{ij})^{\zeta}$. 
If $\eta_i^{\zeta}=\eta_j^{\zeta}$, 
then we have $\eta_i=(\eta_i^{\zeta})^{\zeta^{-1}}=(\eta_j^{\zeta})^{\zeta^{-1}}=\eta_j$. 
Thus elements $\eta_1^{\zeta},\cdots, \eta_t^{\zeta}$ are all of distinct eigen values of 
$(\alpha_{ij})^{\zeta}$ with multiplicities $m_1,\cdots,m_t$. 
Let $(\beta_{ij})=(\alpha_{ij})^{\zeta}$. 
Let $j$ be an integer with $1\leq j\leq r$. 
Since $\gamma^{\zeta}x_j\gamma^{-\zeta}=\prod_{i=1}^rx_i^{\beta_{ij}}$, 
it follows that 
\begin{align*}
\psi(\gamma x_j \gamma^{-1})
&=
\psi(\gamma)\psi(x_j)\psi(\gamma)^{-1}
\\
&
=
(x\gamma^{\zeta})\left(\prod_{k=1}^rx_k^{\delta_{kj}}\right)(\gamma^{-\zeta} x^{-1})
\\
&
=
x\left(\prod_{k=1}^r\gamma^{\zeta}x_k^{\delta_{kj}}\gamma^{-\zeta} \right)x^{-1}
\\
&
=
x\left(\prod_{k=1}^r\prod_{i=1}^rx_i^{\beta_{ik}\delta_{kj}} \right)x^{-1}
\\
&
=
\prod_{i=1}^rx_i^{\sum_{i=1}^r\beta_{ik}\delta_{kj}}.
\end{align*}
On the other hand, 
it also follows that
\begin{align*}
\psi(\gamma x_j \gamma^{-1})
&
=
\psi\left(\prod_{k=1}^rx_k^{\alpha_{kj}}\right)
\\
&
=
\prod_{k=1}^r\psi(x_k)^{\alpha_{kj}}
\\
&
=
\prod_{k=1}^r\prod_{i=1}^rx_i^{\delta_{ik}\alpha_{kj}}
\\
&
=
\prod_{i=1}^rx_i^{\sum_{k=1}^r\delta_{ik}\alpha_{kj}}.
\end{align*}
Thus it holds that $(\alpha_{ij})^{\zeta}(\delta_{ij})=(\delta_{ij})(\alpha_{ij})$, 
and hence we have
$$
\prod_{i=1}^t(T-\eta_{i}^{\zeta})^{m_i}
=
\prod_{i=1}^t(T-\eta_{i})^{m_i},
$$
here we denote by $T$ a variable, 
because the multiplicity of $\eta_i^{\zeta}$ as an eigen value of the matrix 
$(\alpha_{ij})^{\zeta}$ is $m_i$ as stated in the above. 
The correspondence 
$\eta_i\mapsto \eta_i^{\zeta}$ defines an action of $\lr{\zeta}$ on 
$\{\eta_1,\cdots,\eta_t\}$ without fixed points. 
Indeed, 
suppose that $\eta_i=\eta_i^{\zeta^c}$ holds for some $i$ with $1\leq i\leq t$ and an integer $c$. 
We remark here that $\eta_i\neq 1$. 
If the order of $\eta_i$ is finite then is a power of $p$, 
hence it holds that $\zeta^c \equiv 1\bmod{p}$. 
Since $\zeta\in \mu_{p-1}$, 
we have $\zeta^c=1$. 
If the order of $\eta_i$ is infinite then it holds that $\zeta^c=1$. 
Thus we have $c\equiv 0\bmod{d}$ in both cases. 
Thus, 
if necessary, 
by rearranging, 
we can see that 
all of distinct eigen values of $(\alpha_{ij})$ can be represented as 
$$
\eta_1,\cdots ,\eta_s,\eta_1^{\zeta},\cdots,\eta_{s}^{\zeta}
,\cdots,
\eta_1^{\zeta^{d-1}},\cdots,\eta_{s}^{\zeta^{d-1}}
$$
with multiplicities
$$
m_1,\cdots, m_s,m_1,\cdots,m_s,\cdots,m_1,\cdots,m_s,
$$
for some $s$. 
Therefore, 
we have 
$r=d(m_1+\cdots +m_s)$. 
\end{proof}

\section{Lemmas}

In this section, 
we give some lemmas.

\begin{lem}\label{lem3}
Let $p$ be a prime number and $F/\Q$ a finite extension. 
Let $K/F$ be a Galois extension. 
Assume that there is at least one prime ${\frak l}$ of $k$ such that $K/F$ is totally ramified at 
${\frak l}$. 
\\
$(1)$ 
Let $M/K$ be an unramified abelian extension such that $M/F$ is a Galois extension. 
Then it holds that $\ga(M/F)\simeq \ga(M/K)\rtimes \ga(K/F)$. 
\\
$(2)$ 
Suppose that $\ga(K/F)$ is isomorphic to $\Z_p$ or 
$\Z/p^r\Z$ for some positive integer $r$. 
Let $\tau$ be a topological generator of $\ga(K/F)$. 
If $K/F$ is ramified at only the prime ${\frak l}$, 
then we have $X_K/(\tau-1)X_K\simeq X_F$. 
\end{lem}

\begin{proof}
$(1)$ 
Let $\ca{T}$ be the inertia subgroup of a prime of $M$ above ${\frak l}$ in $\ga(M/F)$. 
Since $K/F$ is totally ramified at ${\frak l}$ and since $\ca{T}\cap \ga(M/K)=1$, 
it holds that 
$$
\ca{T}\simeq \ga(M/K)\ca{T}/\ga(M/K) \simeq  \ga(K/F).
$$
Hence we have $\ga(M/F)=\ga(M/K)\ca{T}$. 
For $x,x'\in \ga(M/K)$ and $t,t'\in \ca{T}$, 
assume that $xt=x't'$. 
Since ${x'}^{-1}x=t't^{-1}\in \ga(M/K)\cap \ca{T}=1$, 
we have $x=x'$ and $t=t'$. 
This shows that 
$$
\ga(M/F)=\ga(M/K)\ca{T}\simeq \ga(M/K)\rtimes \ga(K/F).
$$
This was what we should prove. 
\\
$(2)$ 
Let $L$ be the maximal intermediate field of $L_K/F$ such that $L/F$ is abelian. 
Since the extensions $L_F/F$ and $K/F$ are abelian, 
it follows that $L_FK\subseteq L$. 
Let $T$ be the inertia subgroup at ${\frak l}$ in $\ga(L/F)$. 
Then the fixed field of $T$ in $L$ is $L_F$. 
Thus we have $\ga(L/L_FK)=T\cap \ga(L/K)=1$ since $L/K$ is unramified, 
and hence $L=L_FK$. 
One can easily see that $\ga(L_K/L)=(\tau-1)X_K$. 
Therefore, 
we have 
$$
X_F\simeq \ga(L_FK/K)=\ga(L/K)\simeq X_K/(\tau-1)X_K.
$$
This completes the proof. 
\end{proof}

Let $F/\Q$ be a finite extension, 
$k/F$ a quadratic field and $J$ the generator of $\ga(k/F)$. 
Let $k_{\infty}^a/k$ be an anti-cyclotomic like $\Z_p$-extension with respect to $k/F$. 

\begin{lem}\label{lem4}
If the prime number $p$ does not split in $k_{\infty}^a/\Q$, 
then $k_{\infty}^a/k$ is totally ramified at the unique prime of $k$ above $p$. 
\end{lem}

\begin{proof}
Suppose that $k_1^a/k$ is an unramified extension. 
By class field theory, 
the Artin map induces an isomorphism $A_k\simeq \ga(L_k/k)$ as $\lr{J}$-modules. 
Since $k_1^a$ is a subfield of $L_k$ and $J$ acts on $\ga(k_1^a/k)$ as $-1$, 
there is a surjective map $A_k^-\to \ga(k_1^a/k)$. 
By our assumption that the prime ${\frak p}$ of $k$ above $p$ does not split in 
$k_{\infty}^a/\Q$, 
${\frak p}$ is inert in $k_1^a/k$. 
Also, 
it holds that $J({\frak p})={\frak p}$. 
This implies that the Artin symbol of ${\frak p}$ in $k_1^a/k$ is trivial, 
thus ${\frak p}$ splits in $k_1^a/k$. 
This is a contradiction. 
Hence $k_{\infty}^a/k$ is totally ramified at the unique prime above $p$. 
\end{proof}

\section{Proof of Theorem \ref{1}}  

From here to end of this article, 
let $p$ be an odd prime number. 
In this section we show a somewhat general result which leads the assertion of Theorem \ref{1}. 
Let $F/\Q$ be a finite extension, 
$k/F$ a finite Galois extension 
and put $\Delta=\ga(k/F)$. 
Let $K/k$ be a $\Z_p$-extension such that $K/F$ is a Galois extension 
and put $\Gamma=\ga(K/k)$. 
Let $\gamma$ be a topological generator of $\Gamma$. 
Then there is $\chi\in {\rm Hom}(\Delta,\mu_{p-1})$ 
such that $\delta_1\gamma\delta_1^{-1}=\gamma^{\chi(\delta)}$ 
for each $\delta\in\Delta$ and an extension $\delta_1\in \ga(K/F)$ of $\delta$. 
Indeed, 
since ${\rm Aut}(\Z_p)=\Z_p^{\times}=\mu_{p-1}\times (1+p\Z_p)$ 
and $\Delta$ is finite, 
there is $\chi\in {\rm Hom}(\Delta,\mu_{p-1})$ such that 
$\Delta$ acts on $\Gamma$ as $\chi$. 

\begin{thm}\label{3}
Let the notations be as above. 
Let $d$ be the order of a finite cyclic group $\chi(\Delta)$. 
Let $s$ be the multiplicity of $T$ of the characteristic polynomial, 
lies in $\Z_p[T]$, 
with respect to the linear map $\gamma-1$ on $\Q_p\otimes_{\Z_p}X_K$. 
Then we have $\lambda\equiv s\bmod{d}$. 
\end{thm}

\begin{proof}
It holds that $X_K$ is a $\ga(K/F)$-module. 
Thus the $\Z_p$-torsion submodule ${\rm Tor}_{\Z_p}X_{K}$ 
of $X_K$ is also a $\ga(K/F)$-submodule. 
This shows that the fixed field of ${\rm Tor}_{\Z_p}X_{K}$ is a Galois extension over $F$. 
Since $\lambda =\dim_{\Q_p}\Q_p\otimes_{\Z_p}X_K$, 
we may assume that $X_K$ is a free $\Z_p$-module of rank $\lambda$. 

For each element $\delta\in \Delta$, 
we denote by $\tilde{\delta}\in \ga(L_K/F)$ an extension of $\delta$. 
Also, 
let $\tilde{\gamma}\in \ga(L_K/k)$ be an extension of $\gamma$. 
Since $\Delta$ acts on $\Gamma$ as $\chi$, 
there is $z_{\delta}\in X$ such that $\tilde{\delta}\tilde{\gamma}\tilde{\delta}^{-1}
=\tilde{\gamma}^{\chi(\delta)}z_{\delta}$. 
Also, 
each element $g\in \ga(L_K/F)$ can be written as 
$g=y\tilde{\gamma}^{a}\tilde{\delta}$ for some $y\in X_K$, 
$a\in \Z_p$ and $\delta\in \Delta$. 
Let $x\in X_K^{\Gamma}$. 
Then 
\begin{align*}
\tilde{\gamma}(g^{-1}xg)\tilde{\gamma}^{-1}
&
=
\tilde{\gamma}(\tilde{\delta}^{-1}\tilde{\gamma}^{-a}y^{-1}xy
\tilde{\gamma}^{a}\tilde{\delta}
)
\tilde{\gamma}^{-1}
\\
&
=
\tilde{\delta}^{-1}\tilde{\delta}
\tilde{\gamma}\tilde{\delta}^{-1}\tilde{\gamma}^{-a}x
\tilde{\gamma}^{a}\tilde{\delta}
\tilde{\gamma}^{-1}
\tilde{\delta}^{-1}\tilde{\delta}
\\
&
=
\tilde{\delta}^{-1}
\tilde{\gamma}^{\chi(\delta)}
z_{\delta}
x
z_{\delta}^{-1}
\tilde{\gamma}^{-\chi(\delta)}
\tilde{\delta}
\\
&
=
\tilde{\delta}^{-1}
\tilde{\gamma}^{\chi(\delta)}
x
\tilde{\gamma}^{-\chi(\delta)}
\tilde{\delta}
\\
&
=
\tilde{\delta}^{-1}
x
\tilde{\delta}
\\
&
=
\tilde{\delta}^{-1}\tilde{\gamma}^{-a}y^{-1}xy
\tilde{\gamma}^{a}\tilde{\delta}
\\
&
=
g^{-1}xg,
\end{align*}
and hence $X_K^{\Gamma}$ is a closed normal subgroup of $\ga(L_K/F)$. 
Suppose that $s>0$. 
By a standard argument of linear algebra, 
the multiplicity of $T$ of the characteristic polynomial of the linear map 
$\gamma-1$ on 
$
X_K/X_K^{\Gamma}
\simeq 
(\gamma-1)X_K
\subseteq X_K
$ 
is less than $s$.  
Thus, 
by repeating the same argument, 
we can find an intermediate field $M$ of $L_K/K$ such that 
$M/F$ is a Galois extension with the property that $\ga(M/K)$ is a free $\Z_p$-module, 
and that the characteristic polynomial of $\gamma-1$ 
on $\ga(L_K/M)$ is $T^s$. 
Put $X=\ga(M/K)$. 
Then it holds that $X^{\Gamma}=0$ and that $X\simeq \Z_p^{\lambda-s}$. 
By lemma \ref{lem3}, 
we have $\ga(M/k)\simeq X\rtimes \Gamma$. 
Let $\delta\in \Delta$ be an element such that the order of $\zeta=\chi(\delta)$ is $d$. 
Let $\tilde{\delta}\in \ga(M/F)$ be an extension of $\delta$ and 
$\psi$ be an automorphism of $\ga(M/k)$ defined by $\psi(g)=\tilde{\delta}g\tilde{\delta}^{-1}$. 
Since $\Delta$ acts on $\Gamma$ as $\chi$, 
there is $x\in X$ such that $\psi(\tilde{\gamma})=x\tilde{\gamma}^{\chi(\delta)}
=
x\tilde{\gamma}^{\zeta}$. 
By lemma \ref{lem2}, 
we have $\lambda -s \equiv 0\bmod{d}$. 
This completes the proof. 
\end{proof}

We show Theorem \ref{1}. 
Let $k$ be a quadratic extension of a finite extension $F/\Q$. 
Let $k_{\infty}^a/k$ be an anti-cyclotomic like $\Z_p$-extension with respect to $k/F$. 
Since $k_{\infty}^a/k$ is totally ramified at the unique prime above $p$ by lemma \ref{lem4}, 
it holds that $A_k\simeq X_{k_{\infty}^a}/(\gamma-1)X_{k_{\infty}^a}$. 
Thus we have $s=0$. 
The fields $k$, 
$F$ and $k_{\infty}^a$ satisfy the conditions of Theorem \ref{3} with $d=[k:F]=2$, 
and therefore we have $\lambda\equiv 0 \bmod{2}$. 
\qed

\section{
Proof of Theorem \ref{2}
}

Let $k$ be a CM-field and $k^+$ the totally real subfield of $k$. 
Let $J$ be the generator of $\ga(k/k^+)$. 
Let $k_{\infty}^a/k$ be an anti-cyclotomic like $\Z_p$-extension with respect to $k/k^+$. 
We remark that $k_{\infty}^a/k$ is totally ramified at the unique prime above $p$. 
For each non-negative integer $n$, 
put $A_n=A_{k_n^a}$. 
Since norm maps $A_m\to A_n$ for each pair $m$ and $n$ of non-negative 
integers with $m\geq n$ are surjective, 
it suffices to show that $A_1$ is not cyclic. 
If $A_0$ is not cyclic then $A_1$ is also not cyclic. 
Suppose that $A_0$ is a non-trivial cyclic group. 
Suppose further that $A_{k^+}=0$. 
Then one has $A_0=A_0^-$. 
Put $G_1=\ga(k_1^a/k)\simeq \Z/p\Z$. 
Following the method of the proof of Theorem $6.1$ of \cite{Kundu-Washington}, 
we show here that $\# A_0 <\# A_1$. 
Let $\tau$ be a generator of $G_1$. 
By lemma \ref{lem3}, 
it holds that 
$$
\# A_0=\# A_1/(\tau-1)A_1=\# A_1^{G_1}.
$$
Suppose that $\# A_0= \# A_1$. 
Then we have $A_1^{G_1}=A_1$, 
and hence $\ga(k/k^+)=\lr{J}$ acts on $A_1$ canonically. 
It holds that the norm map $A_1^-\to A_0^-=A_0$ is surjective. 
Further since it holds that 
$$
\# A_0 =\# A_0^-\leq \# A_1^- \leq \# A_1 =\# A_0,
$$
we have $A_1^-=A_1$. 

For $i=0$ or $1$, 
let $I_i$, 
$P_i$, 
$C_i$ and 
$E_i$ be the ideal group, 
the principal ideal group, 
the ideal class group and the unit group of $k_i^a$. 
From the exact sequence 
$$
0\to P_1\to I_1 \to C_1\to 0,
$$
we have an exact sequence 
$$
I_1^{G_1}\to C_1^{G_1}\to \hat{H}^{-1}(G_1,P_1)\to 0
$$
of abelian groups. 
One also sees that 
$$
(I_1^{G_1})_p^-\to A_1^-\to \hat{H}^{-1}(G_1,P_1)_p^-\to 0
$$
is exact since $\Z_p\otimes_{\Z} C_1^{G_1}=A_1^{G_1}=A_1$. 
From the exact sequence
$$
0\to E_1\to (k_1^a)^{\times}\to P_1 \to 0,
$$
we have an exact sequence 
$$
0\to \hat{H}^{-1}(G_1,P_1)\to \hat{H}^0(G_1,E_1).
$$
Let $\mu(k)$ be the group of all roots of unity of $k$, 
and let $E_0^+$ be the unit group of $k^+$. 
Then, 
as Hasse's unit index, 
it is known that $[E_0:\mu(k)E_0^+]=1$ or $2$. 
From assumptions that $p$ is odd and $k$ contains no primitive $p$th roots of unity,  
it follows that
$$
(E_0)_p^-
=
(\mu(k)E_0^+/E_0^+)_p^-
=
0.
$$
Since $\hat{H}^0(G_1,E_1)_p^-$ is a quotient of $(E_0)_p^-$, 
it holds that $\hat{H}^0(G_1,E_1)_p^-=0$. 
Hence we have 
$
\hat{H}^{-1}(G_1,P_1)_p^-=0
$, 
and then it turns out that $(I_1^{G_1})_p^-\to A_1^-$ is surjective. 
Let ${\frak p}$ be the prime of $k$ above $p$, 
and ${\frak p}_1$ be the prime of $k_1^a$ above ${\frak p}$. 
It holds that $I^{G_1}=I_0\lr{{\frak p}_1}$. 
There is the following exact sequence
$$
0\to I_0\to I_1^{G_1} \to \lr{{\frak p}_1}I_0/I_0\to 0.
$$
Since $J({\frak p})={\frak p}$, 
it holds that $J({\frak p}_1)={\frak p}_1$, 
and hence we have $(\lr{{\frak p}_1}I_0/I_0)_p^-=0$. 
Thus it holds that $(I_1^{G_1})_p^-=(I_0)_p^-$. 
This shows that the lifting map $A_0=A_0^-\to A_1^-=A_1$ is surjective. 
However, 
the composition of the norm map $A_1\to A_0$ and the lifting map $A_0\to A_1$ 
is the $p$th power map because of $A_1=A_1^{G_1}$, 
thus the lifting map $A_0\to A_1$ never be surjective. 
This is contradiction. 
Therefore, 
we have $\# A_0< \# A_1$. 
Hence $G_1$ acts on $A_1$ non-trivially. 

By lemma \ref{lem3}, 
we have $\ga(L_{k_1^a}/k)=A_1\rtimes G_1$. 
Let $\tilde{J}\in \ga(L_{k_1^a}/k^+)$ be an extension of $J$. 
Let $\phi$ be an automorphism of $\ga(L_{k_1^a}/k)$ 
defined by $\phi(g)=\tilde{J}g\tilde{J}^{-1}$ for $g\in \ga(L_{k_1^a}/k)$. 
Since $J$ acts as $-1$ on $\ga(k_1^a/k)$, 
it holds that $\phi(\tau)\equiv \tau^{-1}\bmod{A_1}$, 
and hence there is $y\in A_1$ such that $\phi(\tau)=y\tau^{-1}$. 
By lemma \ref{lem1}, 
we can conclude that such an automorpshism $\phi$ does not exist 
if $A_1$ is cyclic. 
Therefore, 
$A_1$ is not cyclic 
\qed

\begin{ack}
The author would like to express his thanks to Professor Manabu Ozaki for telling the article 
\cite{Carroll-Kisilevsky}, 
and for giving valuable comments about contents of this article. 
The research of this article was partly supported by 
JSPS KAKENHI Grant number 24K06669.
\end{ack}

\end{document}